\documentclass[12pt]{amsart}
\usepackage{amsmath,amsfonts,amssymb}
\numberwithin{equation}{section}

\newtheorem{theorem}{Theorem}[section]
\newtheorem{lemma}[theorem]{Lemma}

\newtheorem{proposition}[theorem]{Proposition}

\newtheorem{definition}[theorem]{Definition}

\DeclareMathOperator{\Tr}{Tr}

\DeclareMathOperator{\End}{End}

\DeclareMathOperator{\Span}{Span} 
\DeclareMathOperator{\Diag}{Diag}

\title [A construction by deformation...]{A construction by deformation of unitary irreducible representations of $SU(1,n)$ and $SU(n+1)$}
\author {Benjamin Cahen}
\address{Universit\'e de Lorraine, Site de Metz, UFR-MIM,
D\'epartement de math\'ematiques,
B\^atiment A,
3 rue Augustin Fresnel, BP 45112,
57073 METZ Cedex 03, France.}
\email{benjamin.cahen@univ-lorraine.fr}


\begin{document}

\maketitle

 \dedicatory{To the memory of my father, Alfred Cahen}

\begin{abstract}  We recover the holomorphic discrete series representations of $SU(1,n)$ as well as some unitary irreducible representations of $SU(n+1)$
by deformation of a minimal realization of $sl(n+1,{\mathbb C})$.
\end{abstract}

\vspace{1cm}

\section {Introduction}  \label{sec:int}

The deformations of Lie algebras were intensively studied
in the years 1960-70 \cite{Ger}, \cite{NR1},  \cite{NR2}, \cite{Le1}
and still remain objects of active research, see for instance \cite{Fia1},
\cite{FiaP} and \cite{Bu}. On the other hand, the deformations of Lie algebra representations have not been studied as systematically, with some notable exeptions, see
\cite{NR3}, \cite{Her}, \cite{Le2} and also \cite{LePin}.

These works lead us to the following considerations. Let $\mathfrak g$ be a (real or complex) Lie algebra and let $\pi$ be a representation of $\mathfrak g$.
\begin{enumerate}
\item If $\pi$ admits non-trivial formal deformations then by taking the deformation parameter to be a real or complex number we can expect to get a one-parameter family of representations of $\mathfrak g$;
\item Conversely, given a one-parameter family of representations of $\mathfrak g$, we can expect to recover it by deformation of the representation obtained by taking the value of the parameter to be zero.\end{enumerate}

Then, constructing (formal) deformations of Lie algebra representations appears as a way to derive a family of
representations from a given one  and then to get many representations from a few ones. In particular, we can hope for applications of deformations to the description
of unitary duals of Lie groups.

However, as pointed in \cite{CaBP}, the existence and classification problems for deformations depend on some Lie algebra cohomology modules which are not easy to compute, see for instance 
\cite{Le2} and \cite{CaBP}.

The goal of the present paper is to show how the above mentionned ideas work on a simple but non-trivial example. More specifically, we aim to recover the discrete series representations of $SU(1,n)$ and also the family of unitary
irreducible representations of $SU(n+1)$ considered in \cite{CaJAM} by deforming 
a so-called minimal realization of $sl(n+1,{\mathbb C})$ \cite{Jo74}.

Let us briefly describe our method. It is known that minimal realizations of  
simple complex Lie algebras are related to minimal (non-trivial) nilpotent coadjoint orbits of the corresponding simple Lie groups 
\cite{Jo76}, \cite{ABC}, \cite{KPW}. But we can easily exhibit a very simple parametrization $\Psi$ of the minimal coadjoint orbit of $SL(n+1,{\mathbb C})$ by
complex coordinates $p_1,p_2,\ldots, p_n,q_1,q_2,\ldots, q_n$ so that the coordinate functions $\tilde {X}$, $X\in sl(n+1,{\mathbb C})$, defined by
\begin{equation*}\tilde {X}(p_1,p_2,\ldots, p_n,q_1,q_2,\ldots, q_n)=\langle \Psi(p_1,p_2,\ldots, p_n,q_1,q_2,\ldots, q_n),X\rangle\end{equation*}
constitute a Lie algebra for the usual Poisson brackets, which is isomorphic to 
$sl(n+1,{\mathbb C})$. Then, denoting by $W$ the classical Weyl correspondence,
the map $\rho_0: X\rightarrow W(i\tilde {X})$ is a representation of $sl(n+1,{\mathbb C})$ which is also a minimal realization of $sl(n+1,{\mathbb C})$
in the sense of \cite{Jo74}. We can thus compute the formal deformations of $\rho_0$, the calculations being simplified by the use of the Moyal star product
as in \cite{ABC} and \cite{CaBP}. By this way, we recover the one-parameter family
of representations of $sl(n+1,{\mathbb C})$ given in \cite{Jo74}. Moreover, by restricting these representations to $su(1,n)$ and $su(n+1)$ and by selecting values of the parameter we obtain the representations of $SU(1,n)$ and $SU(n+1)$
mentionned above.

We would like to emphasize the fact that, despite of technicalities, our method is very simple: by deforming a given representation of $sl(n+1,{\mathbb C})$ (the minimal realization), we get a family of representations of $sl(n+1,{\mathbb C})$
which gives in turn, by restriction an integration, the desired representations of
$SU(1,n)$ and $SU(n+1)$.

Moreover, we could hope for applications of this method to the description of representations of Lie algebras (in particular of unitary dual of simple Lie groups) in more
general situations (some examples can already be found in \cite{LePin} and \cite{CaBP}).

This paper is organized as follows. In Section \ref{sec:2} and Section \ref{sec:3}, we describe the holomorphic discrete series representations of $SU(1,n)$ and the family of unitary
irreducible representations of $SU(n+1)$ which was introduced in \cite{CaJAM} as an analogue to the holomorphic discrete series representations of $SU(1,n)$ and we compute their differentials which can be extended to representations of $sl(n+1, {\mathbb C})$. Section \ref{sec:4} is devoted to some generalities on (formal) deformations of Lie algebra homomorphisms and, in Section \ref{sec:5}, we recall the Moyal star product and the Weyl correspondence \cite{Fo}, \cite{Vo}. In Section \ref{sec:6}, we show how a symplectic chart of the minimal nilpotent coadjoint orbit of $sl(n+1, {\mathbb C})$ naturally leads to a minimal realization of $sl(n+1, {\mathbb C})$.
In Section \ref{sec:7}, we compute the first cohomology module corresponding to the deformation of the minimal realization and then we derive the desired representations of $SU(1,n)$ and $SU(n+1)$ in Section \ref{sec:8}. In particular, by this way we can recover all the irreducible unitary representations of $SU(2)$.

\section {Discrete series representations of $SU(1,n)$}  \label{sec:2}

The group $SU(1,n)$ consists of
all complex $(n+1)\times (n+1)$ matrices $g$ with determinant $1$
such that 
\begin{equation*} g^{\ast} \left(\begin{matrix}-1&0\\ 0&I_n\end{matrix} \right)\,g=\left(\begin{matrix}-1&0\\ 0&I_n\end{matrix} \right)\end{equation*}
where $g^{\ast}=\bar {g}^t$ denotes the conjugate transpose of $g$.

The group $G$ acts holomorphically on the unit ball
\begin{equation*}{\mathbb B}=
 \{z=(z_1,z_2,\ldots , z_n)\in {\mathbb C}^n:\Vert
z\Vert^2 :=\vert z_1\vert^2+\vert z_2\vert^2+\cdots +\vert
z_n\vert^2 =1\}\subset {\mathbb C}^n \end{equation*}
by fractional linear transformations. Indeed, if $g\in G$ is of the form
\begin{equation*}g= \left(\begin{matrix}a&b\\ c&d\end{matrix}\right)\end{equation*}
with matrices $a(1\times 1)$, $b(1\times n)$, $c(n\times 1)$ and
$d(n\times n)$, then the action of $g$ on $z\in {\mathbb B}$ is defined by
\begin{equation*}g\cdot z=(a+bz^t)^{-1}(c+dz^t)^t.\end{equation*}
Here the subscript $t$ denotes transposition.

A $G$-invariant measure on $\mathbb B$ is
\begin{equation*}d\mu(z)=(1-\Vert z\Vert^2)^{-(n+1)}dx_1dy_1\ldots dx_ndy_n.\end{equation*}
Here we use the notation  $z=(x_1+iy_1,x_2+iy_2,\ldots
, x_n+iy_n)$ where $x_k,y_k\in {\mathbb R}$ for $k=1,2,\ldots ,n$.

For each integer $m>n$, we can consider the Hilbert space ${\mathcal H}_m$ of
all holomorphic functions $f$ on $\mathbb B$ such that
\begin{equation*}\Vert f \Vert_m^2 :=\tfrac{m(m-1)\ldots (m-n)}{{\pi}^{n}}\int_{\mathbb B}\,\vert f(z)\vert^2\, (1-\Vert
z\Vert^2)^{m-n-1}d\mu(z)\, <
\infty. \end{equation*}

Let us consider the representation $\sigma_m$ of $SU(1,n)$ on ${\mathcal
H}_m$ defined by
\begin{equation*}(\sigma_m(g)\,f)(z)=(bz^t+a)^{-m}f(g^{-1}\cdot z),\quad
g^{-1}=\left(\begin{matrix}a&b\\
c&d\end{matrix}\right).\end{equation*}
Then $\sigma_m$ lies in the discrete series representation of $SU(1,n)$, see for instance \cite{Kn}.

The Lie algebra $su(1,n)$ of $SU(1,n)$ consists of all matrices of the
form
\begin{equation*} \left(\begin{matrix}i\alpha& b\\ b^{\ast}& A \end{matrix}\right)\end{equation*}
where $\alpha \in {\mathbb R}$, $b \in {\mathbb C}^n$ and $A$ is an
anti-Hermitian $n\times n$ matrix (that is, $A^{\ast}=-A$) such that
$i\alpha+\Tr (A)=0$.

We extend the differential $d\sigma_m$
of $\sigma_m$ to a representation of $sl(n+1,{\mathbb C})=su(n+1)^{\mathbb C}$ also denoted by $d\sigma_m$. We have
\begin{equation*}(d\sigma_m(X)f)(z)=m(\beta z^t+\alpha)f(z)+df(z)((\alpha+\beta z^t)z-(\gamma+\delta z^t)^t)\end{equation*}
for
\begin{equation*}X= \left(\begin{matrix}\alpha&\beta\\ \gamma&\delta\end{matrix} \right)\in sl(n+1,{\mathbb C})\end{equation*}
with matrices $\alpha(1\times 1)$, $\beta(1\times n)$, $\gamma(n\times 1)$ and
$\delta(n\times n)$.

In order to give more explicit formulas for $d\sigma_m$, let us introduce the following
basis of $sl(n+1,{\mathbb C})$. For $1\leq i,\,j\leq n+1$, we write $E_{ij}$ for the matrix whose $ij$-th entry is
$1$ and all of the other entries are $0$. Then the matrices
$H_{k}=E_{k+1 k+1}-E_{11}\quad (1\leq k\leq n)$ form a basis for the Cartan 
subalgebra
${\mathfrak h}$ of $sl(n+1,{\mathbb C})$ consisting of all diagonal matrices of
$sl(n+1,{\mathbb C})$ and, obviously, the matrices
$H_k\quad  (1\leq i\leq n)$ and $E_{ij}\quad (1\leq i\not= j\leq n+1)$ form
a basis for $sl(n+1,{\mathbb C})$. Then we have
\begin{align*}(d\sigma_m(H_k)f)(z)=&-mf(z)-z_k\frac{\partial f}{\partial z_k}-\sum_{j=1}^nz_j\frac{\partial f}{\partial z_j}\\
(d\sigma_m(E_{1 k+1})f)(z)=&mz_kf(z)+z_k\sum_{j=1}^n z_j\frac{\partial f}{\partial z_j}\\
(d\sigma_m(E_{ k+1 1})f)(z)=&-\frac{\partial f}{\partial z_k}\\
(d\sigma_m(E_{i+1j+1})f)(z)=&-z_j\frac{\partial f}{\partial z_i}\\
\end{align*}
for $1\leq k\leq n$ and $1\leq i\not=j\leq n$.

More generally, for each $\lambda \in {\mathbb C}$, let us consider
the representation $\rho^{\lambda}$ of $sl(n+1,{\mathbb C})$ on the space
$\mathcal P$ of all complex polynomials on ${\mathbb C}^n$ defined by

\begin{align*}(\rho^{\lambda}(H_k)f)(z)=&-{\lambda}f(z)-z_k\frac{\partial f}{\partial z_k}-\sum_{j=1}^nz_j\frac{\partial f}{\partial z_j}\\
(\rho^{\lambda}(E_{1 k+1})f)(z)=&{\lambda}z_kf(z)+z_k\sum_{j=1}^n z_j\frac{\partial f}{\partial z_j}\\
(\rho^{\lambda}(E_{ k+1 1})f)(z)=&-\frac{\partial f}{\partial z_k}\\
(\rho^{\lambda}(E_{i+1j+1})f)(z)=&-z_j\frac{\partial f}{\partial z_i}\\
\end{align*}
for $1\leq k\leq n$ and $1\leq i\not=j\leq n$.

The following result shows how one can recover $(\sigma_m, {\mathcal H}_m)$
from the representations $\rho^{\lambda}$, ${\lambda}\in {\mathbb C}$.

\begin{proposition} \label{prop:21} Let $\langle \cdot,\cdot \rangle_{\lambda}$ be scalar product on $\mathcal P$ for which the operators $\rho^{\lambda}(X)$, $X\in sl(n+1,{\mathbb C})$, are skew-adjoint, that is, such that
\begin{equation} \label{eq:skew}\langle \rho^{\lambda}(X)f_1 ,f_2 \rangle_{\lambda}+ \langle f_1 , \rho^{\lambda}(X)f_2 \rangle_{\lambda}=0 \end{equation} for each $X\in sl(n+1,{\mathbb C})$ and each $f_1,f_2\in {\mathcal P}$. Then we have $\lambda \notin
-{\mathbb N}$ and there exists a constant $C>0$ such that
\begin{equation*} \langle z^p,z^q \rangle_{\lambda}=C\delta_{pq}\frac {p!}{\lambda(\lambda+1)\ldots (\lambda +\vert p\vert -1)}\end{equation*}
for each $p,q\in {\mathbb N}^n$.

If, moreover, $\rho^{\lambda}$ can be integrated to a representation $\sigma^{\lambda}$ of $SU(1,n)$ on the Hilbert space ${\mathcal H}^{\lambda}$
which is the completion of $\mathcal P$ for the norm associated with  $\langle \cdot,\cdot \rangle_{\lambda}$, then $\lambda$ is an integer $m$ and $\sigma^{\lambda}$ is unitarily equivalent to $\sigma_m$. \end{proposition}

\begin{proof} Let $(e_1,e_2,\ldots ,e_n)$ be the canonical basis of ${\mathbb C}^n$. For each $p\in {\mathbb N}^n$, let $\alpha(p):=\langle z^p,z^p \rangle_{\lambda}$ and let $C:=\alpha (0)$. Applying Eq. \ref{eq:skew}
to $X=H_k$, $k=1,2,\ldots, n$, $f_1=z^p$ and $f_2=z^q$ where $p\not= q$,
we get $\langle z^p,z^q \rangle_{\lambda}=0$. Also, applying Eq. \ref{eq:skew} to $X=E_{1k+1}+E_{k+11}$ and $f_1=f_2=z^p$, we obtain the relation
\begin{equation*}(\lambda +\vert p \vert )\alpha (p+e_k)=(p_k+1)\alpha (p)\end{equation*} for each $p\in {\mathbb N}^n$. This shows that 
$\lambda \notin
-{\mathbb N}$ and, by induction, we obtain
\begin{equation} \label{eq:alpha} \alpha (p)=C\frac {p!}{\lambda(\lambda+1)\ldots (\lambda +\vert p\vert -1)}\end{equation}
for each $p\in {\mathbb N}^n$. Then we have proved the first assertion of the proposition. For the second assertion, note that 
\begin{equation*}\sigma^{\lambda}(\exp (itH_1))1=\exp(t\rho^{\lambda}(iH_1))1=e^{-i\lambda t}\end{equation*} and
\begin{equation*}\exp (itH_1)=\Diag(e^{-i\lambda t},e^{i\lambda t},1,\ldots,1).\end{equation*} Thus $\lambda$ must be an integer $m$. Finally, taking into account Eq. \ref{eq:alpha} and the fact that ${\mathcal H}_m$ has orthonormal basis
\begin{equation*}f_p(z)=\left(\frac{(m+\vert p\vert-1) !}{ (m-1) !\,p !}\right)^{1/2}z^p,\quad p\in {\mathbb N}^n,\end{equation*} see for instance \cite{CaJAM},
we see that $f\rightarrow C^{1/2}f$ is a (unitary) intertwining operator
between $\sigma_m$ and $\sigma^m$.
\end{proof}

\section {Unitary irreducible representations of $SU(n+1)$}  \label{sec:3}

Here we consider a family of representations of $SU(n+1)$ indexed by an integer $m \geq 1$ which is analogous to the discrete series of $SU(1,n)$. In \cite{CaJAM}, we showed that this family can be contracted to the
unitary irreducible representations of the Heisenberg group of dimension $2n+1$
as the holomorphic discrete series representations of $SU(1,n)$.

The group $SU(n+1)$ consists of
all complex $(n+1)\times (n+1)$ matrices $g$ with determinant $1$
such that $ g^{\ast}\,g=I_{n+1}$. Here we write the
elements of the group $SU(n+1)$ as block matrices
\begin{equation*}g= \left(\begin{matrix}a&b\\ c&d\end{matrix} \right)\end{equation*}
with matrices $a(1\times 1)$, $b(1\times n)$, $c(n\times 1)$ and
$d(n\times n)$.

The group $SU(n+1)$ acts naturally on the projective space ${\mathbb
P}_n({\mathbb C})$ and this action induces an holomorphic action (defined
almost everywhere) of $SU(n+1)$ on ${\mathbb C}^n$ by fractional linear
transformations
\begin{equation*} g\cdot z=(a+bz^t)^{-1}(c+dz^t)^t,\quad \quad g= \left(\begin{matrix}a&b\\
c&d\end{matrix}\right).\end{equation*} 

For each integer $m\geq 1$, let ${\mathcal P}_m$ be the space of
all complex polynomial functions on ${\mathbb C}^n$ of degree $\leq m$. 
We endow
${\mathcal P}_m$ with the Hilbert product
\begin{equation*}\langle f_1,f_2\rangle_m:=\tfrac{ (m+1)\ldots (m+n)} {{\pi}^{n}}\int_{{\mathbb C}^n}\, f_1(z)\overline { f_2(z)}
(1+\Vert z\Vert^2)^{-m-n-1}
dx_1dy_1\ldots dx_ndy_n.\end{equation*}

Let $\pi_m$ be the representation of $SU(n+1)$ on ${\mathcal
P}_m$ defined by
\begin{equation*}(\pi_m(g)\,f)(z)=(bz^t+a)^{m}f(g^{-1}\cdot z),\quad
g^{-1}=\left(\begin{matrix}a&b\\
c&d\end{matrix}\right).\end{equation*}
We can easily verify that $\pi_m$ is unitary.

The Lie algebra $su(n+1)$ of $SU(n+1)$ consists of all matrices of the
form
\begin{equation*} \left(\begin{matrix}i\alpha& b\\ -b^{\ast}& A \end{matrix}\right)\end{equation*}
where $\alpha \in {\mathbb R}$, $b \in {\mathbb C}^n$ and $A$ is an
anti-Hermitian $n\times n$ matrix such that
$i\alpha+\Tr (A)=0$. 

The differential $d\pi_m$
of $\pi_m$ can be extended to a representation of $sl(n+1,{\mathbb C})$ also denoted by $d\pi_m$. We have
\begin{equation*}(d\pi_m(X)f)(z)=-m(\beta z^t+\alpha)f(z)+df(z)((\alpha+\beta z^t)z-(\gamma+\delta z^t)^t)\end{equation*}
where 
\begin{equation*}X= \left(\begin{matrix}\alpha&\beta\\ \gamma&\delta\end{matrix} \right)\end{equation*}
with matrices $\alpha(1\times 1)$, $\beta(1\times n)$, $\gamma(n\times 1)$ and
$\delta(n\times n)$.

More precisely, we have
\begin{align*}(d\pi_m(H_k)f)(z)=&mf(z)-z_k\frac{\partial f}{\partial z_k}-\sum_{j=1}^nz_j\frac{\partial f}{\partial z_j}\\
(d\pi_m(E_{1 k+1})f)(z)=&-mz_kf(z)+z_k\sum_{j=1}^n z_j\frac{\partial f}{\partial z_j}\\
(d\pi_m(E_{ k+1 1})f)(z)=&-\frac{\partial f}{\partial z_k}\\
(d\pi_m(E_{i+1j+1})f)(z)=&-z_j\frac{\partial f}{\partial z_i}\\
\end{align*}
for $1\leq k\leq n$ and $1\leq i\not=j\leq n$.

We can easily see that $d\pi_m$-hence $\pi_m$- is irreducible. Indeed, let ${\mathcal V}$
be a nonzero subspace of ${\mathcal P}_m$ which is invariant under $d\pi_m(X)$ for each
$X\in sl(n+1,{\mathbb C})$. Then there exists at least one nonzero element $f$ in ${\mathcal V}$. Thus, by applying
the operators $d\pi_m(E_{ k+1 1})$ to $f$, we get $1\in {\mathcal V}$
and by applying the operators $d\pi_m(E_{1 k+1})$ and $d\pi_m(E_{i+1j+1})$ to $1$ we see that ${\mathcal V}={\mathcal P}_m$.

Let us denote by $\epsilon_k$, $1\leq k\leq n$, the linear form on $\mathfrak h$ defined by
\begin{equation*}\epsilon_k: \Diag(a_1,a_2,\ldots,a_{n+1}) \rightarrow a_k. \end{equation*}
It is well-known that the root system of $sl(n+1,{\mathbb C})$ relative to $\mathfrak h$ is 
\begin{equation*}\Delta=\{\epsilon_i-\epsilon_j\,:\,1\leq i,j\leq n+1\},\end{equation*}
see for instance \cite{GoodW}. The ordering on $\Delta$ is usually taken so that the positive roots are $\epsilon_i-\epsilon_j$ $(1\leq i<j\leq n+1)$. In this context, we
can verify that $d\pi_m$ has highest weight $m\epsilon_1$ and highest weight vector $f=z_n^m$.

\section {Generalities on deformations}  \label{sec:4}

In this section, we recall some definitions and results of deformation theory. The material
of this section is essentially taken from \cite{NR3}, \cite{Her}, \cite{Le2}, see also \cite{Gui} and \cite{CaBP}.

Let $\mathfrak g$ be a Lie algebra over $\mathbb C$ and let $A$ be an associative algebra
over $\mathbb C$ with unit element $1$. Then $A$ is also a Lie algebra for the commutator
$[a,b]:=ab-ba$. Let $\varphi: {\mathfrak g}\rightarrow A$ be a Lie algebra homomorphism.

\begin{definition} \begin{enumerate} 
\item A formal deformation of $\varphi$ is a formal series $\Phi=\sum_{k\geq 0}t^k\Phi_k$ where $\Phi_0=\varphi$ and $\Phi_k$ is a linear map from $\mathfrak g$ to $A$ for each $k \geq 1$, such that 
\begin{equation}\label{eqDef}\Phi ([X,Y])=[\Phi(X),\Phi(Y)]\end{equation} 
for each $X$ and $Y$ in $\mathfrak g$. Here we have extended the bracket of $A$ to formal series by bilinearity.

\item Two formal deformations $\Phi$ and $\Psi$ of $\varphi$ are said to be equivalent if there exists a series $a=1+ta_1+t^2a_2+\ldots \in A[[t]]$ such that for each $X\in {\mathfrak g}$, we have 

\begin{equation}a^{-1}\Phi(X)a=\Psi(X).\end{equation}
\end{enumerate}\end{definition}

The study of the formal deformations of $\varphi$ naturally leads us to consider
the structure of $\mathfrak g$-module on $A$ defined by $X\cdot a=[\varphi(X),a]$ for
$X\in {\mathfrak g}$ and $a\in A$ and the Chevalley-Eilenberg cohomology of $\mathfrak g$
with values in the $\mathfrak g$-module $A$. Indeed, denoting by $\partial$ the corresponding cobord
operator, we immediately see that Eq. \ref{eqDef} is equivalent to the fact that for each $n\geq 0$ and each $X,Y \in {\mathfrak g}$, we have
\begin{align*}(\partial \Phi_n)[X,Y]:=&[\varphi (X),\Phi_n(Y)]+[\Phi_n(X),\varphi(Y)]-\Phi_n([X,Y])\\=&-\sum_{k=1}^{n-1}[\Phi_k(X),\Phi_{n-k}(Y)].\end{align*}

In particular, we see that if such a deformation $\Phi$ exists then $\Phi_1$ is a $1$-cocycle.

We have the following result, see for instance \cite{Her}, Section III and \cite{Le2}, Section I.

\begin{proposition}\begin{enumerate}
\item If we have $H^2({\mathfrak g},A)=(0)$ then, for each $1$-cocycle $\alpha:
{\mathfrak g}\rightarrow A$, there exists a formal deformation $\Phi$ such that $\Phi_1=\alpha$.
\item If we have $H^1({\mathfrak g},A)=(0)$ then each formal deformation $\Phi$ of $\varphi$ is equivalent to $\varphi$.  \end{enumerate} \end{proposition}

In \cite{CaBP}, we proved the following result.

\begin{proposition} Assume that $H^1({\mathfrak g},A)$ is one-dimensional and that there exists a formal deformation $\Phi$ of $\varphi$ such the class of $\Phi_1$ generates 
$H^1({\mathfrak g},A)$. For each sequence $c=(c_k)_{k\geq 1}$ of complex numbers, consider the formal series
$S_c(t):=\sum_{k\geq 1}c_kt^k$ and the formal deformation $\Phi^c$ of $\varphi$
defined by $\Phi^c(X)=\sum_{r\geq 0}S_c(t)^r\Phi_r(X)$ for each $X\in {\mathfrak g}$.

Then the map $c\rightarrow \Phi^c$ is a bijection from the set of all sequences $c=(c_k)_{k\geq 1}$ of $\mathbb C$ onto the set of all equivalence classes of formal deformations of $\varphi$. \end{proposition}

Note that the preceding definitions and results can be applied to the particular case of a representation $\varphi$ of $\mathfrak g$ in a complex vector space $V$, since $\varphi$
is also a Lie algebra homomorphism from $\mathfrak g$ to $\End(V)$, or, more generally, to
a subalgebra $A$ of $\End(V)$.

\section {Weyl correspondence and Moyal star product}  \label{sec:5}
 
Here we first recall the Moyal star product, see for instance \cite{AC}. Take coordinates
$(p,q)$ on ${\mathbb R}^{2n}\cong  {\mathbb R}^{n}\times  {\mathbb R}^{n}$ and let
$x=(p,q)$. Then one has $x_i=p_i$ for $1\leq i \leq n$ and $x_i=q_{i-n}$ for $n+1\leq i\leq 2n$. For $u,v\in C^{\infty}({\mathbb R}^{2n})$, define $P^0(u,v):=uv$,
\begin{equation*}P^1(u,v):=\sum_{k=1}^n\left( \frac{\partial u}{\partial p_k}
\frac{\partial v}{\partial q_k}-\frac{\partial u}{\partial q_k}
\frac{\partial v}{\partial p_k}\right)=\sum_{1\leq i,j\leq n}\Lambda^{ij}{\partial_{x_i}}u
{\partial_{x_j}}v \end{equation*}
(the Poisson brackets) and, more generally, for $l\geq 2$,
\begin{equation*}P^l(u,v):=\sum_{1\leq i_1,\ldots, i_l,j_1,\ldots,j_l\leq n}
\Lambda^{i_1j_1}\Lambda^{i_2j_2}\cdots \Lambda^{i_lj_l}\partial^l_{x_{i_1}\ldots x_{i_l}}u \,\partial^l_{x_{j_1}\ldots x_{j_l}}v.\end{equation*}

Then the Moyal product $\ast_M$ is the following formal deformation of the pointwise multiplication of $C^{\infty}({\mathbb R}^{2n})$
\begin{equation*}u \ast_M v:=\sum_{l\geq 0}\frac{t^l}{l!}P^l(u,v)\end {equation*}
where $t$ is a formal parameter.
Moreover, the corresponding Moyal brackets are given by
\begin{equation*}[u, v]_{\ast_M }:=\frac{1}{2t}(u \ast_M v-v \ast_M u)
=\sum_{l\geq 0}\frac{t^{2l}}{(2l+1)!}P^{2l+1}(u,v).\end {equation*}

Now, we restrict $\ast_M$ to polynomials on ${\mathbb R}^{2n}$ and take $t=-i/2$.
Then we get an associative product $\ast$ on polynomials which we denote by $\ast$. This product corresponds to the composition of operators in the usual Weyl quantization procedure as we will explain below.

The Weyl correspondence on ${\mathbb R}^{2n}$ is defined as follows, see \cite{CombR},
\cite{Fo}, \cite{Ho}. 
For each $f$ in the Schwartz space ${\mathcal S}({\mathbb
R}^{2n})$, we define the operator $ W(f)$ acting on the Hilbert
space $L^2({\mathbb R}^{n})$ by \begin{equation*}W(f)\varphi
(p)={(2\pi)}^{-n}\,\int_{{\mathbb R}^{2n}}\, e^{isq} \,f( p+(1/2)s,
q)\,\varphi (p+s)\,ds\,dq.
\end{equation*}

As it is well-known, that the Weyl calculus can be extended to much larger
classes of symbols (see for instance \cite{Ho}). In particular, if $f(p,q)=u(p)q^{\alpha}$ where
$u\in C^{\infty}({\mathbb R}^n)$ then we have
\begin{equation}\label{eq:W} W(f)\varphi (p)=\left(i\frac {\partial }{ \partial
s}\right)^{\alpha} \left( u(p+(1/2)s)\,\varphi (p+s) \right) \Bigl
\vert _{s=0},\end{equation}  see \cite{Vo}. For instance, if
$f(p,q)=u(p)$ then $ W(f)\varphi (p)=u(p)\,\varphi (p)$ and if $
f(p,q)=u(p)q_k$ then
\begin{equation}\label{eq:Wbis}W(f)\varphi (p)=i \bigl((1/2){\partial_k  u }(p)\,\varphi (p)
+u(p) {\partial _k  \varphi }(p)\bigr). \end{equation}

Moreover, we have $W(f_1\ast f_2)=W(f_1)W(f_2)$ for each functions $f_1,f_2$
on ${\mathbb R}^{2n}$
of the form $u(p)q^{\alpha}$, in particular for polynomials, see \cite{Fo}, p. 103.

Note also that, since the map $W$ and the product $\ast$ on polynomials can be defined in a purely
algebraic way, see Eq. \ref{eq:W}, we can extended them to the polynomials
in complex variables $p,q$ without any modification.

\section {Minimal realization}  \label{sec:6}

In \cite{ABC}, a general method for constructing minimal realizations
of semisimple complex Lie algebras from minimal coadjoint orbits was introduced. In the particular case
of the Lie algebra ${\mathfrak g}:=sl(n+1, {\mathbb C})$ of $G:=SL(n+1, {\mathbb C})$,
this method goes as follows.

First, we can identify the dual ${\mathfrak g}^{\ast}$ of ${\mathfrak g}$ with ${\mathfrak g}$ by means of the
bilinear form on ${\mathfrak g}$ defined by $\langle X, Y\rangle:= \Tr(XY)$. In this identification, the coadjoint action of $G$ corresponds to the adjoint action of $G$ and the
coadjoint orbits to the adjoint orbits.

This is a simple exercice to show that the minimal (non trivial) nilpotent (co)adjoint orbit
$\mathcal O$ of $G$ consists of all rank one matrices of ${\mathfrak g}$.

Now, let us consider the map $\Psi$ from ${\mathbb C}^{2n}$ to ${\mathcal O}':={\mathcal O}\cup (0)$ defined by
\begin{equation*} \Psi(p,q):=\begin{pmatrix}
-\sum_{j=1}^np_jq_j&q_1&\dots&q_n\\
-p_1\sum_{j=1}^np_jq_j&p_1q_1&\dots&p_1q_n\\
\hdotsfor[3]{4}\\
-p_n\sum_{j=1}^np_jq_j&p_nq_1&\dots &p_nq_n\\
\end{pmatrix}.\end{equation*}

Then the image of $\Psi$ is a dense open subset of ${\mathcal O}'$.

For each $X\in {\mathfrak g}$, let us denote by $\tilde {X}$ the corresponding coordinate function on
${\mathbb C}^{2n}$:
\begin{equation*}\tilde {X}(p,q):=\langle  \Psi(p,q), X\rangle. \end{equation*}

\begin{proposition}\begin{enumerate}
\item For each $X,Y\in  {\mathfrak g}$, we have
\begin{equation*}[\tilde {X},\tilde {Y}]_{\ast}=\{\tilde {X},\tilde {Y}\}=\tilde{[X,Y]}.
\end{equation*}
\item The map $\rho_0:X\rightarrow W(i\tilde {X})$ is a representation of ${\mathfrak g}$
in $\mathcal P$.
\end{enumerate}\end{proposition}

\begin{proof} (1) Let $X$ and $Y$ in ${\mathfrak g}$. The equation $\{\tilde {X},\tilde {Y}\}=\tilde{[X,Y]}$ can be verified by a direct computation. On the other hand, since $\tilde {X}$ and $\tilde {Y}$ are polynomials of degree $\leq 1$ in the variables $q_1,q_2,\ldots,q_n$, we have $P^k(\tilde {X},\tilde {Y})=0$ for each $k\geq 3$, hence we get $[\tilde {X},\tilde {Y}]_{\ast}=\{\tilde {X},\tilde {Y}\}$.

(2) Let $X$ and $Y$ in ${\mathfrak g}$. By (1), we have
\begin{equation*}(i\tilde {X})\ast (i\tilde {Y})-(i\tilde {Y})\ast (i\tilde {X})=i
\tilde{[X,Y]}.\end{equation*}
Then, by the remark at the end of Section \ref{sec:5}, we get $[W(i\tilde {X}),W(i\tilde {Y})]=W(i
\tilde{[X,Y]})$ hence the result.
\end{proof}

The representation $\rho_0$ is a minimal realization of ${\mathfrak g}$, that is, a realization
of ${\mathfrak g}$ as Lie algebra of differential operators acting on functions of $n$ variables with $n$ minimal, see \cite{Jo74}.

Note that \begin{equation*}\Span\{E_{n+1 2}, \ldots, E_{n+1 n}, E_{21}, \ldots, E_{n+1 1}\}\end{equation*}
is a Heisenberg Lie algebra of dimension $2n-1$ with central element $E_{n+1 1}$,
the only non trivial brackets being $[E_{n+1 k}, E_{k 1}]=E_{n+1 1}$ for $k=2,\ldots, n$.
Then $\Psi$ was chosen so that $\tilde{E}_{n+1 k}=p_{k-1}q_n, \tilde{E}_{k 1 }=q_{k-1}$
(for $k=2,\ldots, n$) and  $\tilde{E}_{n+1 1}=q_n$. In fact, these conditions determine $\Psi$ uniquely.

Now, we aim to study the deformations of $\rho_0$. By using the map $f\rightarrow W(if)$
this is equivalent to studying the deformations of the Lie algebra homomorphism
$X\rightarrow \Phi_0(X):=\tilde{X}$ from $\mathfrak g$ to $M:={\mathbb C}[p,q]$ endowed
with $[\cdot,\cdot]_{\ast}$.

As explained in Section \ref{sec:4}, we endow $M$ with the $\mathfrak g$-module structure defined by $X\cdot f:=[\tilde{X},f]_{\ast}$ and then consider the corresponding Chevalley-Eilenberg cohomology.

\section {Determination of $H^1({\mathfrak g},M)$}  \label{sec:7}

Recall that $H^1({\mathfrak g},M)$ is the quotient space $Z^1({\mathfrak g},M)/
B^1({\mathfrak g},M)$ where $Z^1({\mathfrak g},M)$ consists of all linear maps
$\varphi: {\mathfrak g}\rightarrow M$ satisfying
\begin{equation} \label{eq:cocy}\partial \varphi (X,Y):=[\tilde{X},\varphi(Y)]_{\ast}+[\varphi(X),\tilde{Y}]_{\ast}-\varphi [X,Y ]=0\end{equation}
(the $1$-cocycles) and $B^1({\mathfrak g},M)$ consists of all maps from ${\mathfrak g}$ to $M$ of the form
$X\rightarrow [\tilde{X},f]_{\ast}$ for $f\in M$ (the $1$-coboundaries).

The aim of this section is to compute $H^1({\mathfrak g},M)$. We begin with the following 'Poincar\'e lemma'.

\begin{lemma} \label{lemPoinc} Let $F_i(q)$, $i=1,2,\ldots, n$ be a family of polynomials in the variable $q=(q_1,q_2,\ldots, q_n)$ such that, for each $i,j=1,2,\ldots, n$, one has $\frac{\partial F_i}{\partial q_j}=\frac{\partial F_j}{\partial q_i}$. Then there exists a polynomial $F(q)$ such that $\frac{\partial F}{\partial q_i}=F_i$ for each $i=1,2,\ldots, n$.\end{lemma}

\begin{proof}By the usual Poincar\'e lemma, the result is true for polynomials in real variables $q_i$ which implies that it is also true for polynomials in complex variables $q_i$.
\end{proof}

\begin{proposition}\label{prop:H} The space $H^1({\mathfrak g},M)$ is one dimensional, generated by the class of the cocycle $\varphi_1$ defined by $\varphi_1(E_{11}-E_{22})=1$, $\varphi_1(E_{kk}-E_{k+1 k+1})=0$ for $k=2,\ldots,n$, $\varphi_1(E_{1 k+1})=p_k$ for $k=1,2,\ldots,n$ and $\varphi_1(E_{ij })=0$ for $i\geq 2$. \end{proposition}

\begin{proof} We have divided the proof into several steps. The method of the proof is quite elementary and consists in transforming progressively a given $1$-cocycle to an equivalent one which is more simple by adding suitable $1$-coboundaries.

Let us consider a $1$-cocycle $\varphi:{\mathfrak g}\rightarrow M$.

1) First we apply Eq. \ref{eq:cocy} to $X=E_{ k+1 1}$ and $Y=E_{ l+1 1}$
for $k,l=1,2,\ldots,n$. Writing $\varphi_k=\varphi(E_{ k+1 1})$ for simplicity, we get $\frac{\partial \varphi_k}{\partial p_l}=\frac{\partial \varphi_l}{\partial p_k}$ for each $k,l=1,2,\ldots,n$. Then, by decomposing each $\varphi_k$ as $\varphi_k=\sum_{\alpha}\varphi_k^{\alpha}(p)q^{\alpha}$ with
the usual multi-index notation, we have $\frac{\partial \varphi_k^{\alpha}}{\partial p_l}=\frac{\partial \varphi_l^{\alpha}}{\partial p_k}$ for each $k,l,\alpha$. 

Thus, by Lemma \ref{lemPoinc}, for each $\alpha$ there exists a polynomial $\varphi^{\alpha}(p)$ such that $\frac{\partial \varphi^{\alpha}}{\partial p_k}=
\varphi_k^{\alpha}$ for each $k=1,2,\ldots,n$. 

Now, let $\phi:=\sum_{\alpha}\varphi^{\alpha}(p)q^{\alpha}$. For each $k=1,2,\ldots,n$, we have
\begin{equation*}[\phi, q_k]_{\ast}=\frac{\partial \phi}{\partial p_k}=\varphi_k.
\end{equation*}

Hence, replacing $\varphi$ by the equivalent $1$-cocycle $\varphi-[\phi, \cdot]_{\ast}$, we can always assume that $\varphi(E_{ k+1 1})=0$ for each
$k=1,2,\ldots,n$.

2) We apply Eq. \ref{eq:cocy} to $X=E_{ k l}$, $k,l\geq 2$, $k\not= l$ and $Y=E_{ j+1 1}$, $j=1,2,\ldots,n$. Taking 1) into account, we can immediately see that $\varphi(E_{ k l})$ is a polynomial in the variables $q_1,q_2,\ldots, q_n$.

3) Similarly, applying Eq. \ref{eq:cocy} to $X\in {\mathfrak h}$ and $Y=E_{j+1 1}$,
we verify that $\varphi(X)$ is a polynomial in the variables $q_1,q_2,\ldots, q_n$.

4) Now, we fix $k=1,2,\ldots,n-1$ and we apply Eq. \ref{eq:cocy} to $X=E_{ n+1 k+1}$ and
$Y=\sum_{j=1}^{n-1}(E_{ n+1 n+1}-E_{ j+1 j+1})$. Write $\varphi_k:=\varphi (E_{n+1 k+1})$ for simplicity and recall that $\varphi_k$ is a polynomial in $q_1,q_2,\ldots, q_n$ by 2). Then we see that there exists a polynomial $u_k(q)$ in $q_1,q_2,\ldots, q_n$ such that
\begin{equation}\label{eq:hom}-n\varphi_k=\sum_{j=1}^{n}q_j\frac{\partial \varphi_k}{\partial q_j}+q_nu_k(q).\end{equation}
Let $\varphi_k= \sum_m\varphi_k^m$ and $u_k= \sum_m u_k^m$ be the decompositions
of $\varphi_k$ and $u_k$ into homogeneous polynomials of degree $m$ in $q_1,q_2,\ldots,q_n$ .
Then Eq. \ref{eq:hom} implies that
\begin{equation*}-n\sum_m\varphi_k^m=\sum_m m\varphi_k^m
+q_nu_{k-1}(q)\end{equation*}
and we conclude that, for each $k=1,2,\ldots,n-1$, there exists a polynomial $\psi_k$ in
$q_1,q_2,\ldots, q_n$ such that $\varphi_k=q_n\psi_k$.

Taking $X=E_{n+1 k+1}$ and $Y=E_{n+1 l+1}$ in Eq. \ref{eq:cocy}  for $k,l=
1,2,\ldots,n-1$, we get $\frac{\partial \psi_k}{\partial q_l}=\frac{\partial \psi_l}{\partial q_k}$ for each $k,l$. This implies the existence of a polynomial $\psi$ in $q_1,q_2,\ldots, q_n$ such that $\psi_k=\frac{\partial \psi}{\partial q_k}$ for each $k=1,2,\ldots,n-1$.

Thus, by replacing $\varphi$ by $\varphi-[\cdot,\psi]_{\ast}$, we are led to the case where
$\varphi (E_{ n+1 k+1})=0$ for each $k=1,2,\ldots,n-1$ and the condition $\varphi (E_{k+1 1})=0$ for each $k=1,2,\ldots,n$ is still satisfied.

5) Let $k=1,2,\ldots,n$, $l=1,2,\ldots,n-1$ with $k\not= l$. By applying Eq. \ref{eq:cocy}
to $X=E_{ k+1 l+1}$ and $Y=E_{ n+1 j+1}$ for $j=1,2,\ldots,n-1$, we see that 
$\varphi (E_{ k+1 l+1})$ only depends on $q_n$. Thus, taking into account the equality
\begin{equation*} [E_{ k+1 l+1},E_{ l+1 k+1}]=E_{ k+1 k+1}-E_{ l+1 l+1}\end{equation*} we get $\varphi(E_{ k+1 k+1}-E_{ l+1 l+1})=0$. Hence, applying
Eq. \ref{eq:cocy} to $X=E_{ k+1 k+1}-E_{ l+1 l+1}$ and $Y=E_{ k+1 l+1}$ we obtain
$\varphi (E_{ k+1 l+1})=0$.

Finally, we apply Eq. \ref{eq:cocy} to $X=E_{ j+1 n+1}$ and $Y=E_{ k+1 j+1}$
and we also obtain $\varphi (E_{ k+1 n+1})=0$.

6) Now, take $X\in {\mathfrak h}$ and $Y=E_{ k+1 j+1}$ in Eq. \ref{eq:cocy}.
Then we see that $\varphi(X)$ only depends on $q_n$.

Let $H_0=E_{ 11}-E_{ 22}\in {\mathfrak h}$. Then we can replace $\varphi$ by $\varphi+
[\cdot,F(q_n)]_{\ast}$ for a suitable polynomial $F(q_n)$ so that $\varphi(H_0)$ is a constant which we denote by $a$.

7) Taking $X=E_{ 12}$ and successively $Y=E_{ k+1 1}$, ($k=1,2,\ldots,n$) and
$Y=E_{ n+1 k}$, ($k=2,\ldots,n-1$) in Eq. \ref{eq:cocy}, we see that 
\begin{equation*}\varphi (E_{ 12})=ap_1+f(q_n)\end{equation*}
where $f(q_n)$ is a polynomial.
Moreover, taking also $X=E_{ 12}$ and $Y=H_0$, we get $-2f(q_n)=q_n
\frac {\partial f}{ \partial
q_n}$ hence $f=0$ and $\varphi (E_{ 12})=ap_1$.

8) Finally, we apply Eq. \ref{eq:cocy} to $X=E_{ 12}$ and $Y=E_{ 2 k+1}$
where $k=2,\ldots,n$ we obtain $\varphi (E_{ 1k+1})=ap_k$.

\end{proof}

\section {Recovering $\sigma_m$ and $\pi_m$} \label{sec:8}

In this section, we retain the notation of the previous sections. Proposition \ref{prop:H} leads us to consider the formal deformations $\Phi$ of $\Phi_0:X\rightarrow {\tilde X}$ such that $\Phi_1=a\varphi_1$ for $a\in {\mathbb C}$. We have the following result.

\begin{proposition} For each $a\in {\mathbb C}$, the map $\Phi_a:{\mathfrak g}\rightarrow M[[t]]$ defined by $\Phi_a(X)={\tilde X}+ta\varphi_1(X)$ is a formal deformation of $\Phi_0$ in $M$.\end{proposition}

\begin{proof}Taking into account that $\varphi_1$ is a $1$-cocycle (see Section \ref{sec:7}), the result follows immediately from the equality $[\varphi_1(X),\varphi_1(Y)]_{\ast}=0$ for $X,Y\in {\mathfrak g}$. \end{proof}

By using the properties of $W$ (see Section \ref{sec:5}), we get the following proposition.

\begin{proposition} \label{prop:81} For each $a\in {\mathbb C}$, let $m(a):=-1/2(a+n+1)$.
Then the map $\rho_a$ defined by 
\begin{equation*}\rho_a(X)=W\left({i\tilde X}+\frac{1}{2}a\varphi_1(X)\right)\end{equation*}
for each $X\in {\mathfrak g}$ is a representation of ${\mathfrak g}$ in $\mathcal P$ and we have

\begin{align*}(\rho_a(H_k)f)(z)=&m(a)f(z)-z_k\frac{\partial f}{\partial z_k}-\sum_{j=1}^nz_j\frac{\partial f}{\partial z_j}\\
(\rho_m(E_{1 k+1})f)(z)=&-m(a)z_kf(z)+z_k\sum_{j=1}^n z_j\frac{\partial f}{\partial z_j}\\
(\rho_a(E_{ k+1 1})f)(z)=&-\frac{\partial f}{\partial z_k}\\
(\rho_a(E_{i+1j+1})f)(z)=&-z_j\frac{\partial f}{\partial z_i}\\
\end{align*}
for $1\leq k\leq n$ and $1\leq i\not=j\leq n$.
\end{proposition} 

\begin{proof} The fact that $\rho_a$ is a representation ${\mathfrak g}$ follows from Proposition \ref{prop:81} and the formulas for $\rho_a$ can be easily verified  by Eq. \ref{eq:Wbis}.\end{proof}

Then we immediately see that by applying Proposition \ref{prop:21} we can recover
the representations $\pi_m$ of Section \ref{sec:2}. In order to recover also the representations $\pi_m$ of Section \ref{sec:3}, we select now
the values of $a$ (or, equivalently, of $m(a)$) for which there exists a non trivial
finite dimensional subspace of ${\mathcal P}$ that is invariant under $\rho_a$.  

\begin{proposition} Let $a\in {\mathbb C}$. Assume that ${\mathcal Q}$ is a 
non-trivial
finite dimensional subspace of ${\mathcal P}$ that is invariant under $\rho_a$.
Then $m(a)$ is a non negative integer, we have ${\mathcal Q}={\mathcal P}_{m(a)}$ and the restriction of $\rho_a$ to ${\mathcal Q}$ coincides with $d\pi_{m(a)}$. \end{proposition} 

\begin{proof} Let $a\in {\mathbb C}$. Let ${\mathcal Q}\not= (0)$ a 
finite dimensional subspace of ${\mathcal P}$ which is invariant under $\rho_a$.
Define $m:=\max \{\deg (f)\,:\,f\in {\mathcal Q}\setminus (0)\}$. Let $f$ be an element of $\mathcal Q$ of degree $m$. Let us decompose $f$ as
$f=\sum_{k=0}^mf_k$ where, for each $k$, $f_k$ is an homogeneous polynomial
of degree $k$. Then we have $f_m\not= 0$ and
\begin{equation*}\rho_a(E_{1 l+1})f=p_l\sum_{k=0}^m(k-m(a))f_k.\end{equation*}
We see that if $m\not= m(a)$, we get a contradiction. Thus we have $m(a)=m$
hence $m(a)$ is a non negative integer and ${\mathcal Q}\subset {\mathcal P}_{m(a)}$. Since ${\mathcal P}_{m(a)}$ is irreducible under the action of $d\pi_
{m(a)}$, see Section \ref{sec:3}, we can conclude that ${\mathcal Q}={\mathcal P}_{m(a)}$.
\end{proof}

Then we have recovered the representations $d\pi_m$ of Section \ref{sec:3}, hence the representations $\pi_m$ by integration. Note that by taking $n=1$
we see that this method gives all the unitary irreducible representations of $SU(2)$.

We conclude by the following remarks.
\begin{enumerate}
\item In \cite{ABC}, minimal realizations of the classical simple complex lie algebras were constructed by using the Moyal star product. Then we can expect applications
of the method of the present paper to the description of some families of unitary irreducible representations of simple Lie groups.

\item Contraction of representations can be interpreted geometrically as a kind
of convergence of coadjoint orbits associated with representations via the Kirillov-Kostant method of orbits, see \cite{Do} and \cite{CaIllinois}. In the other hand,
deformation of representations can be considered as the inverse process as contraction \cite{FiaM}. It would be then interesting to find a geometrical interpretation of deformation of representations. \end{enumerate}

\end{document}